\documentclass{amsart}
\usepackage{amscd}
\usepackage{amsmath,empheq}
\usepackage{amsfonts}
\usepackage{amssymb}
\usepackage{mathrsfs}

\usepackage[all]{xy}

\newtheorem{theorem}{Theorem}[section]

\newtheorem{prop}[theorem]{Proposition}
\newtheorem{remark}{Remark}[section]

\newtheorem{claim}{Claim}

\newenvironment{proof-sketch}{\noindent{\bf Sketch of Proof}\hspace*{1em}}{\qed\bigskip}

\everymath{\displaystyle}

\newcommand{\RR}{\mathbb R}
\newcommand{\NN}{\mathbb N}

\renewcommand{\leq}{\leqslant}

\renewcommand{\geq}{\geqslant}

\DeclareMathOperator*{\essinf}{ess\,inf}

\begin{document}
%\hfill\today\bigskip

\title[Robin double-phase problems with singular and superlinear terms]{Robin double-phase problems with \\ singular and superlinear terms}

%%%%%%%%%%%%%%%%%%%%%%%%%%%%%%%%%%%%%%%%%%%%%%%%%%%%%%%%%%%%%%%%%%%%%%%
\author[N.S. Papageorgiou]{N.S. Papageorgiou}
\address[N.S. Papageorgiou]{ Department of Mathematics,
National Technical University,
				Zografou Campus, 15780 Athens, Greece \& Institute of Mathematics, Physics and Mechanics, 1000 Ljubljana, Slovenia}
\email{\tt npapg@math.ntua.gr}

\author[V.D. R\u{a}dulescu]{V.D. R\u{a}dulescu}
\address[V.D. R\u{a}dulescu]{Faculty of Applied Mathematics, AGH University of Science and Technology, 30-059 Krak\'ow, Poland 
\& Institute of Mathematics, Physics and Mechanics, 1000 Ljubljana, Slovenia
\& Department of Mathematics, University of Craiova, 200585 Craiova, Romania}
\email{\tt radulescu@inf.ucv.ro}

\author[D.D. Repov\v{s}]{D.D. Repov\v{s}}
\address[D.D. Repov\v{s}]{Faculty of Education and Faculty of Mathematics and Physics, University of Ljubljana \& Institute of Mathematics, Physics and Mechanics, 1000 Ljubljana, Slovenia}
\email{\tt dusan.repovs@guest.arnes.si}


\begin{abstract}
We consider a nonlinear Robin problem driven by the sum of
$p$-Laplacian
and
$q$-Laplacian
(i.e. the $(p,q)$-equation).
In the reaction there are
competing
effects of a singular term and
 a parametric perturbation $\lambda f(z,x)$, which is Carath\'eodory and  $(p-1)$-superlinear
 at
  $x\in\RR,$ without satisfying the Ambrosetti-Rabinowitz condition. Using variational tools,
   together with truncation and comparison techniques, we prove a bifurcation-type result describing the changes in the set of positive solutions as the
   parameter
    $\lambda>0$ varies.
\end{abstract}

%\date{\today}

\maketitle

\section{Introduction}

Let $\Omega\subseteq\RR^N$ be a bounded domain with a $C^2$-boundary $\partial\Omega$. In this paper, we study the following nonlinear Robin problem
\begin{equation}
	\left\{
		\begin{array}{ll}
			-\Delta_p u(z)-\Delta_q u(z) + \xi(z) u(z)^{p-1} = u(z)^{-\gamma} + \lambda f(z,u(z))\ \mbox{in}\ \Omega,\\
			\frac{\partial u}{\partial n_{pq}} + \beta(z) u^{p-1}=0\ \mbox{on}\ \partial\Omega,\ u>0,\ \lambda>0,\ 0<\gamma<1,\ 1<q<p.
		\end{array}
	\right\}\tag{$P_{\lambda}$}\label{eqp}
\end{equation}

For every $r\in (1,\infty),$
 we denote by $\Delta_r$ the $r$-Laplace differential operator defined by
$$
\Delta_r u={\rm div}\,(|Du|^{r-2}Du)\ \mbox{for all}\ u\in W^{1,r}(\Omega).
$$

The differential operator of \eqref{eqp} is the sum of
 $p$-Laplacian and
  $q$-Laplacian. Such an operator is not homogeneous and
  it
   appears in the mathematical models of various physical processes. We mention the works of Cherfils \& Ilyasov \cite{1} (reaction-diffusion systems) and Zhikov \cite{18} (elasticity theory). The potential function $\xi\in L^\infty(\Omega)$ satisfies
    $\xi(z)\geq0$ for almost all $z\in\Omega$. In the reaction
   (the right-hand side of \eqref{eqp}),
   we have the combined effects of two nonlinearities of different nature. One nonlinearity is the singular term $u^{-\gamma}$ and the other nonlinearity is the parametric term $\lambda f(z,x)$, where $f(z,x)$ is a Carath\'eodory function (that is, for all $x\in\RR,$ the mapping $z\mapsto f(z,x)$ is measurable and for almost all $z\in\Omega,$ the mapping $ x\mapsto f(z,x)$ is continuous), which exhibits $(p-1)$-superlinear growth near $+\infty$ but without satisfying the usual in such cases Ambrosetti-Rabinowitz condition (the AR-condition for short). In the boundary condition, $\frac{\partial u}{\partial n_{pq}}$ denotes the conormal derivative corresponding to the $(p,q)$-Laplace differential operator. Then according to the nonlinear Green's identity (see Gasinski \& Papageorgiou \cite[p. 210]{2}), we have
$$
\frac{\partial u}{\partial n_{pq}} = (|Du|^{p-2}Du + |Du|^{q-2}Du,n)\ \mbox{for all}\ u\in C^1(\overline\Omega),
$$
with $n(\cdot)$ being the outward unit normal on $\partial\Omega$. The boundary coefficient $\beta\in C^{0,\alpha}(\partial\Omega)$ (with $0<\alpha<1$) satisfies $\beta(z)\geq0$ for all $z\in\partial\Omega$.

In the past, nonlinear singular problems were studied only in the context of Dirichlet equations driven by the $p$-Laplacian (a homogeneous differential operator). We mention the works of Giacomoni, Schindler \& Taka\v c \cite{5}, Papageorgiou, R\u adulescu \& Repov\v{s} \cite{10, 11},  Papageorgiou \& Smyrlis \cite{13}, Papageorgiou \& Winkert \cite{14},
and
Perera \& Zhang \cite{16}.
Nonlinear elliptic problems with unbalanced growth have been studied recently by Papageorgiou, R\u adulescu and Repov\v{s} \cite{prrzamp, prrpams, prrblms}. Double-phase transonic flow problems with variable growth have been considered by Bahrouni, R\u adulescu and Repov\v{s} \cite{bahr}.
A comprehensive study of semilinear singular problems can be found in the book of Ghergu \& R\u adulescu \cite{4}.

Using variational methods based on the critical point theory together with suitable truncation and comparison techniques, we prove a bifurcation type result, describing in a precise way the dependence of the set of positive solutions of \eqref{eqp} on the parameter. So, we produce a critical parameter value $\lambda^*>0$ such that for all $\lambda\in(0,\lambda^*),$ problem \eqref{eqp} has at least two positive solutions, for $\lambda=\lambda^*$ problem \eqref{eqp} has at least one positive solution and for $\lambda>\lambda^*$ there are no positive solutions for problem \eqref{eqp}.

\section{Mathematical background and hypotheses}
Let $X$ be a Banach space. By $X^*$ we denote the topological dual of $X$. Given $\varphi\in C^1(X,\RR)$, we say that $\varphi(\cdot)$ satisfies the ``C-condition", if the following property holds
$$
\begin{array}{ll}
\mbox{``Every sequence}\ \{u_n\}_{n\geq1}\subseteq X\ \mbox{such that} \\
\{\varphi(u_n)\}_{n\geq1}\subseteq\RR\ \mbox{is bounded} \
\mbox{and}\ (1+||u_n||)\varphi'(u_n)\rightarrow0\ \mbox{in}\ X^*\ \mbox{as}\ n\rightarrow\infty,\\
\mbox{admits a strongly convergent subsequence."}
\end{array}
$$

This is a compactness type condition on the functional $\varphi$, which leads to the minimax theory of the critical values of $\varphi(\cdot)$.

The two main spaces in the analysis of problem \eqref{eqp} are the Sobolev space $W^{1,p}(\Omega)$ and the Banach space $C^1(\overline\Omega)$. By $||\cdot||$ we denote the norm on the Sobolev space $W^{1,p}(\Omega)$. We have
$$
||u||=\left[||u||^p_p + ||Du||^p_p\right]^\frac{1}{p}\ \mbox{for all}\ u\in W^{1,p}(\Omega).
$$

The Banach space $C^1(\overline\Omega)$ is ordered with positive (order) cone given by
$$
C_+=\{u\in C^1(\overline\Omega):u(z)\geq0\ \mbox{for all}\ z\in\overline\Omega\}.
$$

This cone has a nonempty interior
$$
D_+ = \{u\in C_+:u(z)>0\ \mbox{for all}\ z\in\overline\Omega\}.
$$

We will also consider another order cone (closed convex cone) in $C^1(\overline{\Omega})$, namely the cone
$$\hat{C}_+=\left\{u\in C^1(\overline\Omega):u(z)\geq0\ \mbox{for all}\ z\in\overline\Omega,\ \frac{\partial u}{\partial n}{|_{\partial\Omega\cap u^{-1}(0)}}\leq 0\right\}.
$$
This cone has a nonempty interior
$${\rm int}\, \hat{C}_+=\left\{u\in C^1(\overline\Omega):u(z)>0\ \mbox{for all}\ z\in\Omega,\ \frac{\partial u}{\partial n}{|_{\partial\Omega\cap u^{-1}(0)}}< 0\right\}.
$$
%Note that $D_+$ is the interior of $C_+$ when the latter is endowed with the weaker $C(\overline\Omega)$-norm topology.

To take care of the Robin boundary condition, we will also use the ``boundary" Lebesgue spaces $L^q(\partial\Omega)$ $(1\leq q\leq\infty)$. More precisely, on $\partial\Omega$ we consider the $(N-1)$-dimensional Hausdorff (surface) measure $\sigma(\cdot)$. Using this measure on $\partial\Omega$ we can define in the usual way the Lebesgue spaces $L^q(\partial\Omega)$ $(1\leq q\leq\infty)$. We know that there exists a continuous, linear map $\gamma_0: W^{1,p}(\Omega)\rightarrow L^p(\partial\Omega)$, known as the ``trace map" such that
$$
\gamma_0(u)=u|_{\partial\Omega}\ \mbox{for all}\ u\in W^{1,p}(\Omega)\cap C(\overline\Omega).
$$

So, the trace map extends the notion of boundary values to all Sobolev functions. We have
$$
{\rm im}\,\gamma_0= W^{\frac{1}{p'},p}(\partial\Omega)\ (\frac{1}{p}+\frac{1}{p'}=1)\ \mbox{and}\ {\rm ker}\,\gamma_0 = W^{1,p}_0(\Omega).
$$

The trace map $\gamma_0$ is compact into $L^q(\partial\Omega)$ for all $q\in \left[1,\frac{(N-1)p}{N-p}\right)$ if $N>p$ and into $L^q(\partial\Omega)$ for all $q\geq1$ if $p\geq N$. In the sequel, for the sake of notational simplicity, we drop the use of the trace map $\gamma_0(\cdot)$. All restrictions of Sobolev functions on $\partial\Omega$ are understood in the sense of traces.

For every $r\in(1,+\infty)$, let $A_r:W^{1,r}(\Omega)\rightarrow W^{1,r}(\Omega)^*$ be defined by
$$
\langle A_r(u),h\rangle = \int_\Omega|Du|^{r-2}(Du,Dh)_{\RR^N}dz\ \mbox{for all}\ u,h\in W^{1,r}(\Omega).
$$

The following proposition summarizes the main properties of this map (see Gasinski \& Papageorgiou \cite{2}).

\begin{prop}\label{prop1}
	The map $A_r(\cdot)$ is bounded (that is, maps bounded sets to bounded sets) continuous, monotone (hence maximal monotone, too) and of type $(S)_+$, that is, if $u_n\xrightarrow{w}u$ in $W^{1,r}(\Omega)$ and $\limsup_{n\rightarrow\infty}\langle A_r(u_n),u_n-u\rangle$, then
	$
	u_n\rightarrow u\ \mbox{in}\ W^{1,r}(\Omega).
	$
\end{prop}

Evidently, the $(S)_+$-property is useful in verifying the C-condition.

Now we introduce the conditions on the potential function $\xi(\cdot)$ and on the boundary coefficient $\beta(\cdot)$.

\smallskip
$H(\xi)$: $\xi\in L^\infty(\Omega)$ and $\xi(z)\geq0$ for almost all $z\in\Omega$.

\smallskip
$H(\beta)$: $\beta\in C^{0,\alpha}(\partial\Omega)$ with $0<\alpha<1$ and $\beta(z)\geq0$ for all $z\in\partial\Omega$.

\smallskip
$H_0$: $\xi\not\equiv0$ or $\beta\not\equiv0$.

\begin{remark}\label{rem1}
	When $\beta\equiv0$ we have the usual Neumann problem.
\end{remark}

The next two propositions can be found in Papageorgiou \& R\u adulescu \cite{9}.

\begin{prop}\label{prop2}
	If $\xi\in L^\infty(\Omega)$, $\xi(z)\geq0$ for almost all $z\in\Omega$ and $\xi\not\equiv0$, then $c_0||u||^p\leq ||Du||^p_p + \int_\Omega \xi(z)|u|^pdz$ for some $c_0>0$ and all $u\in W^{1,p}(\Omega)$.
\end{prop}

\begin{prop}\label{prop3}
	If $\beta\in L^\infty(\partial\Omega),\ \beta(z)\geq0$ for $\sigma$-almost all $z\in\partial\Omega$ and $\beta\not\equiv0$, then $c_1||u||^p\leq ||Du||^p_p + \int_{\partial\Omega}\beta(z)|u|^pd\sigma$ for some $c_1>0$
	and all $u\in W^{1,p}(\Omega)$.
\end{prop}

In what follows, let $\gamma_p:W^{1,p}(\Omega)\rightarrow\RR$ be defined by
$$
\gamma_p(u) = ||Du||^p_p + \int_\Omega\xi(z)|u|^pdz + \int_{\partial\Omega}\beta(z)|u|^pd\sigma\ \mbox{for all}\ u\in W^{1,p}(\Omega).
$$

If hypotheses $H(\xi), H(\beta), H_0$ hold, then from Propositions \ref{prop2} and \ref{prop3} we
can
 infer that
\begin{equation}\label{eq1}
	c_2||u||^p \leq \gamma_p(u)\ \mbox{for some}\ c_2>0\ \mbox{and all}\ u\in W^{1,p}(\Omega).
\end{equation}

As we
have  already mentioned in the introduction, our approach involves also truncation and comparison techniques. So, the next strong comparison principle, a slight variant of Proposition 4 of Papageorgiou \& Smyrlis \cite{13}, will be useful.

\begin{prop}\label{prop4} If $\hat\xi\in L^\infty(\Omega)$ with $\hat\xi(z)\geq0$ for almost all $z\in\Omega, h_1, h_2\in L^\infty(\Omega)$,
$$
0<c_3\leq h_2(z)-h_1(z)\ \mbox{for almost all}\ z\in\Omega,
$$
and the functions $u_1,u_2\in C^1(\overline\Omega)\backslash\{0\}, u_1\leq u_2, u_1^{-\gamma}, u_2^{-\gamma}\in L^\infty(\Omega)$ satisfy
$$
\begin{array}{ll}
-\Delta_p u_1 - \Delta_q u_1 + \hat\xi(z) u_1^{p-1} - u_1^{-\gamma}=h_1\ \mbox{for almost all}\ z\in\Omega,\\
-\Delta_p u_2 - \Delta_q	 u_2 + \hat\xi(z) u_2^{p-1} - u_2^{-\gamma}=h_2\ \mbox{for almost all}\ z\in\Omega,
\end{array}
$$
then $u_2-u_1\in {\rm int}\,\hat{C}_+$.
\end{prop}

Consider a Carath\'eodory function $f_0:\Omega\times\RR\rightarrow\RR$ satisfying
$$
|f_0(z,x)|\leq a_0(z)[1+|x|^{r-1}]\ \mbox{for almost all}\ z\in\Omega \ \mbox{and all}\ x\in\RR,
$$
with $a_0\in L^\infty(\Omega)$ and $1<r\leq p^*=\left\{\begin{array}{ll}\frac{Np}{N-p}&\mbox{if}\ p<N\\+\infty &\mbox{if}\ N\leq p\end{array}\right.$ (the critical Sobolev exponent corresponding to $p$).

We set $F_0(z,x)=\int^x_0 f_0(z,s)ds$ and consider the $C^1$-functional $\varphi_0:W^{1,p}(\Omega)\rightarrow\RR$ defined by
$$
\varphi_0(u)=\frac{1}{p}\gamma_p(u) + \frac{1}{q}||Du||^q_q - \int_\Omega F_0(z,u)dz\ \mbox{for all}\ u\in W^{1,p}(\Omega)\ \mbox{(recall that $q<p$)}.
$$

The next proposition can be found in Papageorgiou \& R\u adulescu \cite{8} and essentially is an outgrowth of the nonlinear regularity theory of Lieberman \cite{6}.

\begin{prop}\label{prop5}
	If $u_0\in W^{1,p}(\Omega)$ is a local $C^1(\overline\Omega)$-minimizer of $\varphi_0$, that is, there exists $\rho_0>0$ such that
	$$
	\varphi_0(u_0)\leq\varphi_0(u_0+h)\ \mbox{for all}\ ||h||_{C^1(\overline\Omega)}\leq\rho_0,
	$$
	then $u_0\in C^{1,\alpha}(\overline\Omega)$ for some $\alpha\in(0,1)$ and $u_0$ is also a local $W^{1,p}(\Omega)$-minimizer of $\varphi_0$, that is, there exists $\rho_1>0$ such that
	$$
	\varphi_0(u_0)\leq\varphi_0(u+h)\ \mbox{for all}\ ||h||\leq\rho_1.
	$$
\end{prop}

	The next fact about ordered Banach spaces is useful in producing upper bounds for functions and can be found in Gasinski \& Papageorgiou \cite[Problem 4.180, p. 680]{3}.

\begin{prop}\label{prop6}
	If $X$ is an ordered Banach space with positive (order) cone $K$,
	$$
	{\rm int}\, K\neq\emptyset\ \mbox{and}\ e\in {\rm int}\, K
	$$
	then for every $u\in X$ we can find $\lambda_u>0$ such that $\lambda_u e-u\in K$.
\end{prop}

Under hypotheses $H(\xi), H(\beta), H_0$, the differential operator $ u\mapsto -\Delta_p u + \xi(z)|u|^{p-2}u$ with the Robin boundary condition, has a principal eigenvalue $\hat\lambda_1(p)>0$ which is isolated, simple and admits the following variational characterization:
\begin{equation}\label{eq2}
	\hat\lambda_1(p)=\inf\left\{\frac{\gamma_p(u)}{||u||^p_p}:u\in W^{1,p}(\Omega),u\neq0\right\}.
\end{equation}

The infimum is realized on the corresponding one-dimensional eigenspace, the elements of which have fixed sign. By $\hat{u}_1(p)$ we denote the positive, $L^p$-normalized (that is, $||\hat{u}_1(p)||_p=1$) eigenfunction corresponding to $\hat\lambda_1(p)>0$. The nonlinear Hopf
 theorem (see, for example, Gasinski \& Papageorgiou \cite[p. 738]{2}) implies that $\hat{u}_1(p)\in D_+$.

Let us fix some basic notation which we will use throughout this work. So, if $x\in\RR$, we set $x^\pm=\max\{\pm x,0\}$ and the for $u\in W^{1,p}(\Omega)$ we define $u^\pm(z)=u(z)^\pm$ for all $z\in\Omega$. We know that
$$
u^\pm\in W^{1,p}(\Omega),\ u=u^+-u^-,\ |u|=u^++u^-.
$$

If $\varphi\in C^1(W^{1,p}(\Omega),\RR)$, then by $K_\varphi$ we denote the critical set of $\varphi$, that is,
$$
K_\varphi = \{u\in W^{1,p}(\Omega):\varphi'(u)=0\}.
$$

Also, if $u,y\in W^{1,p}(\Omega)$, with $u\leq y$, then we define
\begin{eqnarray*}
	&&[u,y]=\{h\in W^{1,p}(\Omega): u(z)\leq h(z)\leq y(z)\ \mbox{for almost all}\ z\in\Omega\},\\
	&&\left[u\right) = \{h\in W^{1,p}(\Omega): u(z)\leq h(z)\ \mbox{for almost all}\ z\in\Omega\},\\
&&{\rm int}_{C^1(\overline\Omega)}[u,y]\ \mbox{= the interior in the $C^1(\overline\Omega)$-norm of $[u,y]\cap C^1(\overline\Omega)$}.
\end{eqnarray*}

Now we introduce our hypotheses on the perturbation $f(z,x)$.

\smallskip
$H(f)$: $f:\Omega\times\RR\rightarrow\RR$ is a Carath\'eodory function such that $f(z,0)=0$ for almost all $z\in\Omega$ and
\begin{itemize}
	\item [(i)] $f(z,x)\leq a(z)(1+x^{r-1})$ for almost all $z\in\Omega$ and all $x\geq0$ with $a\in L^\infty(\Omega), p<r<p^*$;
	\item [(ii)] if $F(z,x)=\int_0^x f(z,s)ds$, then $\lim_{x\rightarrow+\infty}\frac{F(z,x)}{x^p}=+\infty$ uniformly for almost all $z\in\Omega$;
	\item [(iii)] there exists $\tau\in((r-p)\max\left\{\frac{N}{p},1\right\},p^*)$ such that
		$$
		0<\hat\beta_0\leq\liminf_{x\rightarrow+\infty}\frac{f(z,x)x-p F(z,x)}{x^\tau}\ \mbox{uniformly for almost all}\ z\in\Omega;
		$$
	\item [(iv)] for every $\vartheta>0$, there exists $m_\vartheta>0$ such that
	$$m_\vartheta\leq f(z,x)\ \mbox{for almost all}\ z\in\Omega \ \mbox{and all}\ x\geq\vartheta;$$
	\item[(v)] for every $\rho>0$ and $\lambda>0$, there exists $\hat{\xi}_{\rho}^{\lambda}>0$ such that for almost all $z\in\Omega$, the function $x\mapsto f(z,x)+\hat\xi^\lambda_\rho x^{p-1}$ is nondecreasing on $[0,\rho]$.
\end{itemize}

\begin{remark}
	Since we are
	looking
	 for positive solutions and the above hypotheses concern the positive semiaxis, without any loss of generality we may assume that
	\begin{equation}\label{eq3}
		f(z,x) = 0\ \mbox{for almost all}\ z\in\Omega \ \mbox{and all}\ x\leq0.
	\end{equation}
\end{remark}

From hypotheses $H(f),(ii),(iii)$ it follows that
$$
\lim_{x\rightarrow+\infty}\frac{f(z,x)}{x^{p-1}} = +\infty\ \mbox{uniformly for almost all}\ z\in\Omega.
$$

Hence, for almost all $z\in\Omega$ the perturbation $f(z,\cdot)$ is $(p-1)$-superlinear near $+\infty$. However, this superlinearity of $f(z,\cdot)$ is not expressed using the well-known AR-condition. We recall that the AR-condition (unilateral version due to \eqref{eq3}) says that there exist $q>p$ and $M>0$ such that
\addtocounter{equation}{-1}
\begin{subequations}
	\begin{align}
		0&<qF(z,x)\leq f(z,x)x\ \mbox{for almost all}\ z\in\Omega \ \mbox{and all}\ x\geq M, \label{eq4a}\\
		0&<\essinf_{\Omega} F(\cdot,M). \label{eq4b}
	\end{align}
\end{subequations}

Integrating \eqref{eq4a} and using \eqref{eq4b}, we obtain the weaker condition
$$
\begin{array}{ll}
	& c_4x^q\leq F(z,x)\ \mbox{for almost all}\ z\in\Omega \ \mbox{all}\ x\geq M,\ \mbox{and some}\ c_4>0, \\
	\Rightarrow & c_4 x^{q-1}\leq f(z,x)\ \mbox{for almost all}\ z\in\Omega \ \mbox{and all}\ x\geq M.
\end{array}
$$

So, the AR-condition dictates an at least $(q-1)$-polynomial growth for $f(z,\cdot)$. Here we replace the AR-condition with hypothesis $H(f)(iii)$ which is less restrictive and permits superlinear nonlinearities with ``slower" growth near $+\infty$. For example the function
$$
f(x)=x^{p-1}\ln(1+x)\ \mbox{for all}\ x\geq0.
$$
(for the sake of simplicity we have dropped the $z$-dependence) satisfies hypotheses $H(f)$, but fails to satisfy the AR-condition.

We introduce the following sets:
$$
\begin{array}{ll}
	\mathcal{L}=\{\lambda>0:\ \mbox{problem \eqref{eqp} has a positive solution}\}, \\
	S_\lambda = \mbox{the set of positive solutions of \eqref{eqp}}.
\end{array}
$$

Also we set
$$
\lambda^*=\sup\mathcal{L}.
$$

\section{Some auxiliary Robin problems}
Let $\eta>0$. First we examine the following auxiliary Robin problem
\stepcounter{equation}
\begin{equation}\label{eq5}
	\left\{
		\begin{array}{ll}
			-\Delta_p u(z) - \Delta_q u(z) + \xi(z)u(z)^{p-1} = \eta\ \mbox{in}\ \Omega,\\
			\frac{\partial u}{\partial n_{pq}} + \beta(z) u^{p-1}=0\ \mbox{on}\ \partial\Omega,\ u>0.
		\end{array}
	\right\}
\end{equation}

\begin{prop}\label{prop7}
	If hypotheses $H(\xi), H(\beta), H_0$ hold, then for every $\eta>0$ problem \eqref{eq5} has a unique solution $\tilde{u}_\eta\in D_+$, the mapping $\eta\mapsto\tilde{u}_\eta$ is strictly increasing (that is, $\eta<\eta'\Rightarrow \tilde{u}_{\eta'}-\tilde{u}_\eta\in {\rm int}\,\hat{C}_+$) and
	$$
	\tilde{u}_\eta\rightarrow0\ \mbox{in}\ C^1(\overline\Omega)\ \mbox{as}\ \eta\rightarrow0^+.
	$$
\end{prop}

\begin{proof}
	Consider the map $V:W^{1,p}(\Omega)\rightarrow W^{1,p}(\Omega)^*$ defined by
	\begin{eqnarray}
		\langle V(u),h\rangle = \langle A_p(u),h\rangle + \langle A_q(u),h\rangle + \int_\Omega \xi(z)|u|^{p-2}uhdz + \int_{\partial\Omega}\beta(z)|u|^{p-2}uhd\sigma \label{eq6}\\
		\mbox{for all}\ u,h\in W^{1,p}(\Omega).\nonumber
	\end{eqnarray}

	Evidently, $V(\cdot)$ is continuous, strictly monotone (hence maximal monotone, too) and coercive (see \eqref{eq1}). Therefore $V(\cdot)$ is surjective (see Gasinski \& Papageorgiou \cite[Corollary 3.2.31, p. 319]{2}). So, we can find $\tilde{u}_\eta\in W^{1,p}(\Omega), \tilde{u}_\eta\neq0$ such that
	$$
	V(\tilde{u}_\eta)=\eta.
	$$

	The strict monotonicity of $V(\cdot)$ implies that $\tilde{u}_\eta$ is unique. We have
	\begin{equation}\label{eq7}
		\langle V(\tilde{u}_\eta),h\rangle = \eta\int_\Omega hdz\ \mbox{for all}\ h\in W^{1,p}(\Omega).
	\end{equation}

	In \eqref{eq7} we choose $h=-\tilde{u}^-_\eta\in W^{1,p}(\Omega)$. Then
	$$
	\begin{array}{ll}
		& c_2||\tilde{u}^-_\eta||^p\leq 0\ \mbox{(see \eqref{eq1})},\\
		\Rightarrow & \tilde{u}_\eta\geq 0,\ \tilde{u}_\eta\neq0.
	\end{array}
	$$

	From \eqref{eq7} we have
	\begin{equation}\label{eq8}
		\left\{
		\begin{array}{ll}
			-\Delta_p \tilde{u}_\eta(z) -\Delta_q\tilde{u}_\eta(z) + \xi(z)\tilde{u}_\eta (z)^{p-1} = \eta\ \mbox{for almost all}\ z\in\Omega,\\
			\frac{\partial\tilde{u}_\eta}{\partial n_{pq}} + \beta(z)\tilde{u}^{p-1}_\eta = 0\ \mbox{on}\ \partial\Omega.
		\end{array}
		\right\}
	\end{equation}

	From \eqref{eq8} and Proposition 7 of Papageorgiou \& R\u adulescu \cite{8} we deduce that
	$$
	\tilde{u}_\eta\in L^\infty(\Omega).
	$$

	Then the nonlinear regularity theory of Lieberman \cite{6} implies that
	$$
	\tilde{u}_\eta\in C_+\backslash\{0\}.
	$$

	From \eqref{eq8} we have
	$$
	\begin{array}{ll}
		& \Delta_p\tilde{u}_\eta(z) + \Delta_q\tilde{u}_\eta(z) \leq ||\xi||_\infty \tilde{u}_\eta(z)^{p-1}\ \mbox{for almost all}\ z\in\Omega,\\
		\Rightarrow & \tilde{u}_\eta\in D_+\ \mbox{(see Pucci \& Serrin \cite[pp. 111, 120]{17})}.
	\end{array}
	$$

	Suppose that $0<\eta_1<\eta_2$ and let $\tilde{u}_{\eta_1},\tilde{u}_{\eta_2}\in D_+$ be the corresponding solutions of problem \eqref{eq5}. We have
	$$
	\begin{array}{ll}
		& -\Delta_p\tilde{u}_{\eta_1} - \Delta_q\tilde{u}_{\eta_1} + \xi(z)\tilde{u}^{p-1}_{\eta_1} = \eta_1<\eta_2 = -\Delta_p\tilde{u}_{\eta_2} - \Delta_q\tilde{u}_{\eta_2} + \xi(z)\tilde{u}_{\eta_2}\\
		& \mbox{for almost all}\ z\in\Omega, \\
		\Rightarrow & \tilde{u}_{\eta_2} - \tilde{u}_{\eta_1}\in {\rm int}\,\hat{C}_+\ \mbox{(see Proposition \ref{prop4})}, \\
		\Rightarrow & \eta\mapsto\tilde{u}_\eta\ \mbox{is strictly increasing from $(0,+\infty)$ into $C^1(\overline\Omega)$}.
	\end{array}
	$$

	Finally, let $\eta_n\rightarrow0^+$ and let $\tilde{u}_n=\tilde{u}_{\eta_n}\in D_+$ be the corresponding solutions of \eqref{eq5}. As before, via Proposition 7 of Papageorgiou \& R\u adulescu \cite{8}, we can find $c_5>0$ such that
	$$
	||\tilde{u}_n||_\infty\leq c_5\ \mbox{for all}\ n\in\NN.
	$$

	Then from Lieberman \cite{6} we infer that there exist $\alpha\in(0,1)$ and $c_6>0$ such that
	$$
	\tilde{u}_n\in C^{1,\alpha}(\overline\Omega),\ ||\tilde{u}_n||_{C^{1,\alpha}(\overline\Omega)}\leq c_6\ \mbox{for all}\ n\in\NN.
	$$

	Exploiting the compact embedding of $C^{1,\alpha}(\overline\Omega)$ into $C^1(\overline\Omega)$, the monotonicity of the sequence $\{\tilde{u}_n\}_{n\geq1}\subseteq D_+$ and that for $\eta=0, u\equiv0$ is the only solution of \eqref{eq5} we obtain
	$$
	\tilde{u}_n\rightarrow0\ \mbox{in}\ C^1(\overline\Omega).
	$$
The proof is now complete.
\end{proof}

Using Proposition \ref{prop7}, we see that we can find $\eta_0>0$ such that
\begin{equation}\label{eq9}
	\eta\leq \tilde{u}_\eta(z)^{-\gamma}\ \mbox{for all}\ z\in\overline\Omega,\
	%\mbox{all}
	\ 0<\eta\leq\eta_0.
\end{equation}

We consider the following purely singular problem
\begin{equation}\label{eq10}
	\left\{
	\begin{array}{ll}
		-\Delta_p u(z) - \Delta_q u(z) + \xi(z) u(z)^{p-1} = u(z)^{-\gamma}\ \mbox{in}\ \Omega, \\
		\frac{\partial u}{\partial n_{pq}} + \beta(z) u^{p-1}=0\ \mbox{on}\ \partial\Omega,\ u>0,\ 0<\gamma<1.
	\end{array}
	\right\}
\end{equation}

In the first place,
 by a solution of \eqref{eq10} we understand a weak solution, that is, a function $u\in W^{1,p}(\Omega)$ such that
$$
\begin{array}{ll}
u^{-\gamma} h\in L^1(\Omega)\ \mbox{and}\ \langle A_p(u),h\rangle + \langle A_q(u),h\rangle + \int_\Omega\xi(z) u^{p-1}hdz + \int_{\partial\Omega}\beta(z)u^{p-1}hd\sigma \\=
 \int_\Omega u^{-\gamma}hdz\ \mbox{for all}\ h\in W^{1,p}(\Omega).
\end{array}
$$

In fact, using the nonlinear regularity theory, we will be able to
establish
 more regularity for the solution of \eqref{eq10}, which in fact,
  is a strong solution (that is, the equation can be interpreted pointwise almost everywhere on $\Omega$).

\begin{prop}\label{prop8}
	If hypotheses $H(\xi), H(\beta), H_0$ hold, then problem \eqref{eq10} admits a unique solution $v\in D_+$.
\end{prop}

\begin{proof}
	Let $\eta\in (0,\eta_0]$ (see \eqref{eq9}) and recall that $\tilde{u}_\eta\in D_+$. So $m_\eta= \min_{\overline\Omega}\tilde{u}_\eta>0$ and
	\begin{eqnarray}
		&& \eta\leq \tilde{u}^{-\gamma}_\eta\leq m^{-\gamma}_\eta\ \mbox{(see \eqref{eq9})}, \nonumber \\
		\Rightarrow && \tilde{u}^{-\gamma}_\eta\in L^\infty(\Omega). \label{eq11}
	\end{eqnarray}

	We consider the following truncation of the reaction in problem \eqref{eq10}:
	\begin{equation}\label{eq12}
		k(z,x) = \left\{
			\begin{array}{ll}
				\tilde{u}_\eta(z)^{-\gamma} & \mbox{if}\ x\leq \tilde{u}_\eta(z) \\
				x^{-\gamma} & \mbox{if}\ \tilde{u}_\eta(z)<x.
			\end{array}
		\right.
	\end{equation}

	This is a Carath\'eodory function. We set $K(z,x)=\int^x_0 k(z,s)ds$ and consider the $C^1$-functional $\Psi:W^{1,p}(\Omega)\rightarrow\RR$ defined by
	$$
	\Psi(u) = \frac{1}{p}\gamma_p(u) + \frac{1}{q}||Du||^q_q - \int_\Omega K(z,u)dz\ \mbox{for all}\ u\in W^{1,p}(\Omega).
	$$

	From \eqref{eq12} and \eqref{eq11}, we see that $\Psi(\cdot)$ is coercive. Also the Sobolev embedding theorem and the compactness of the trace map, imply that $\Psi(\cdot)$ is sequentially weakly lower semicontinuous. So, we can find $v\in W^{1,p}(\Omega)$ such that
	\begin{eqnarray}
		& \Psi(v) = \inf\{\Psi(u): u\in W^{1,p}(\Omega)\}, \nonumber \\
		\Rightarrow & \Psi'(v) = 0, \nonumber \\
		\Rightarrow & \langle A_p(v),h\rangle + \langle A_q(v),h\rangle + \int_\Omega\xi(z)|v|^{p-2}vhdz + \int_{\partial\Omega}\beta(z)|v|^{p-2}vhd\sigma = \nonumber \\
		& \int_\Omega k(z,v) hdz\ \mbox{for all}\ h\in W^{1,p}(\Omega). \label{eq13}
	\end{eqnarray}

	In \eqref{eq13} we choose $(\tilde{u}_\eta -v)^+\in W^{1,p}(\Omega)$. Then
	\begin{eqnarray}
		&& \langle A_p(v),(\tilde{u}_\eta-v)^+\rangle + \langle A_q(v),(\tilde{u}_\eta-v)^+\rangle + \int_\Omega\xi(z)|v|^{p-2}v(\tilde{u}_\eta-v)^+dz + \nonumber \\
		&& \int_{\partial\Omega}\beta(z)|v|^{p-2}v(\tilde{u}_\eta-v)^+d\sigma = \int_\Omega\tilde{u}^{-\gamma}_\eta(\tilde{u}_\eta-v)^+dz\ \mbox{(see \eqref{eq12})} \nonumber \\
		& \geq & \int_\Omega\eta(\tilde{u}_\eta-v)^+dz\ \mbox{(see \eqref{eq9} and recall that $0<\eta\leq\eta_0$)} \nonumber \\
		& = & \langle A_p(\tilde{u}_\eta),(\tilde{u}_\eta-v)^+\rangle + \langle A_q(\tilde{u}_\eta),(\tilde{u}_\eta-v)^+\rangle + \int_\Omega\xi(z)\tilde{u}^{p-1}_\eta(\tilde{u}_\eta-v)^+dz + \nonumber \\
		&& \int_{\partial\Omega}\beta(z)\tilde{u}^{p-1}_\eta(\tilde{u}_\eta-v)^+d\sigma\ \mbox{(see Proposition \ref{prop7})}, \nonumber \\
		& \Rightarrow & \tilde{u}_\eta\leq v. \label{eq14}
	\end{eqnarray}

	Then from \eqref{eq12}, \eqref{eq13}, \eqref{eq14} we obtain
	\begin{eqnarray}
		\left\{
			\begin{array}{ll}
				-\Delta_p v(z) -\Delta_q v(z) + \xi(z)v(z)^{p-1} = v(z)^{-\gamma}\ \mbox{for almost all}\ z\in\Omega, \\
				\frac{\partial v}{\partial n_{pq}} + \beta(z)v^{p-1}=0\ \mbox{on}\ \partial\Omega
			\end{array}
		\right\}\label{eq15}\\
		\mbox{(see Papageorgiou \& R\u adulescu \cite{7})}. \nonumber
	\end{eqnarray}

	From \eqref{eq14} we have $v^{-\gamma}\leq\tilde{u}^{-\gamma}_\eta\in L^\infty(\Omega)$ (see \eqref{eq11}). So, from \eqref{eq15} and \cite{8} we have $v\in L^\infty(\Omega)$. Then the nonlinear regularity theory of Lieberman \cite{6} implies that $v\in C_+$.
	Hence
	it follows
	from \eqref{eq14}
	 that
	$$
	v\in D_+.
	$$
	
	Next, we show that this positive solution is unique. To this end, let $\hat{v}\in W^{1,p}(\Omega)$ be another positive solution of \eqref{eq10}. Again we have $\hat{v}\in D_+$. Then
	\begin{eqnarray*}
		&& \langle A_p(v),(\hat{v}-v)^+ \rangle + \langle A_q(v),(\hat{v}-v)^+\rangle + \int_\Omega\xi(z)v^{p-1}(\hat{v}-v)^+dz + \\
		&& \int_{\partial\Omega}\beta(z)v^{p-1}(\hat{v}-v)^+d\sigma \\
		& = & \int_\Omega v^{-\gamma}(\hat{v}-v)^+dz \\
		& \geq & \int_\Omega \hat{v}^{-\gamma}(\hat{v}-v)^+dz \\
		& = & \langle A_p(\hat{v}),(\hat{v}-v)^+\rangle + \langle A_q(\hat{v}),(\hat{v}-v)^+\rangle + \int_\Omega\xi(z)\hat{v}^{p-1}(\hat{v}-v)^+dz + \\
		&& \int_{\partial\Omega}\beta(z)\hat{v}^{p-1}(\hat{v}-v)^+d\sigma \\
		& \Rightarrow & \hat{v}\leq v.
	\end{eqnarray*}

	Interchanging the roles of $v$ and $\hat{v}$ in the above argument, we obtain
	$$
	\begin{array}{ll}
		& v\leq \hat{v},\\
		\Rightarrow & v=\hat{v}.
	\end{array}
	$$

	This proves the uniqueness of the positive solution of the purely singular problem \eqref{eq10}.
\end{proof}

Next, we consider the following nonlinear Robin problem
\begin{equation}\label{eq16}
	\left\{
		\begin{array}{ll}
			-\Delta_p u(z) - \Delta_q u(z) + \xi(z) u(z)^{p-1} = v(z)^{-\gamma}+1\ \mbox{in}\ \Omega, \\
			\frac{\partial u}{\partial n_{pq}} + \beta(z) u^{p-1}=0\ \mbox{on}\ \partial\Omega,\ u>0.
		\end{array}
	\right\}
\end{equation}

\begin{prop}\label{prop9}
	If hypotheses $H(\xi), H(\beta), H_0$ hold, then problem \eqref{eq16} admits a unique solution $\overline{u}\in D_+$ and $v\leq\overline{u}$.
\end{prop}

\begin{proof}
	We know that $v^{-\gamma}\in L^\infty(\Omega)$ (see \eqref{eq11} and \eqref{eq14}). Then the existence and uniqueness of the solution $\overline{u}\in W^{1,p}(\Omega)\backslash\{0\},\overline{u}\geq0$ of \eqref{eq16} follow from the surjectivity and strict monotonicity of the map $V(\cdot)$ (see the proof of Proposition \ref{prop7}). The nonlinear regularity theory and the nonlinear Hopf's theorem imply that $\overline{u}\in D_+$.

	Moreover, we have
	\begin{eqnarray*}
		&& \langle A_p(\overline{u}),(v-\overline{u})^+ \rangle + \langle A_q(\overline{u}),(v-\overline{u})^+\rangle + \int_\Omega\xi(z)\overline{u}^{p-1}(v-\overline{u})^+dz + \\
		&& \int_{\partial\Omega}\beta(z)\overline{u}^{p-1}(v-\overline{u})^+d\sigma \\
		& = & \int_\Omega [v^{-\gamma}+1](v-\overline{u})^+dz\ \mbox{(see \eqref{eq16})} \\
		& \geq & \int_\Omega v^{-\gamma}(v-\overline{u})^+dz \\
		& = & \langle A_p(v),(v-\overline{u})^+\rangle + \langle A_q(v,(v-\overline{v})^+)\rangle + \int_\Omega\xi(z)v^{p-1}(v-\overline{v})^+dz + \\
		&& \int_{\partial\Omega}\beta(z)v^{p-1}(v-\overline{v})^+d\sigma \\
		& \Rightarrow & v\leq\overline{u}.
	\end{eqnarray*}
The proof is now complete.
\end{proof}

\section{Positive solutions}
In this section we prove the bifurcation-type theorem described in the Introduction.

\begin{prop}\label{prop10}
	If hypotheses $H(\xi), H(\beta), H_0, H(f)$ hold, then $\mathcal{L}\neq\emptyset$ and $S_\lambda\subseteq D_+$.
\end{prop}

\begin{proof}
	Let $v\in D_+$ be the unique positive solution of the auxiliary problem \eqref{eq10} (see Proposition \ref{prop8}) and $\overline{u}\in D_+$ the unique solution of \eqref{eq16} (see Proposition \ref{prop9}). We know that $v\leq \overline{u}$ (see Proposition \ref{prop9}). Since $\overline{u}\in D_+$, hypothesis $H(f)(i)$ implies that
	$$
	0\leq f(z,\overline{u}(z))\leq c_7\ \mbox{for some}\ c_7>0 \ \mbox{and almost all}\ z\in\Omega.
	$$

	So, we can find $\lambda_0>0$ small such that
	\begin{equation}\label{17}
		0 \leq \lambda f(z,\overline{u}(z)) \leq 1\ \mbox{for almost all } z\in \Omega \ \mbox{and all}\ 0 < \lambda\leq\lambda_0.
	\end{equation}

	We consider the following truncation of the reaction in problem \eqref{eqp}

	\begin{equation}\label{eq18}
	\vartheta_\lambda(z,x) = \left\{
		\begin{array}{ll}
			v(z)^{-\gamma} + \lambda f(z,v(z)) & \mbox{if}\ x < v(z) \\
			x^{-\gamma} + \lambda f(z,x) &\mbox{if}\ v(z) \leq x \leq\overline{u}(z) \\
			\overline{u}(z)^{-\gamma} + \lambda f(z,\overline{u}(z)) & \mbox{if}\ \overline{u}(z) < x.
		\end{array}
	\right.
	\end{equation}

	This is a Carath\'eodory function. We set $\theta_\lambda(z,x) = \int_{0}^x \vartheta_\lambda (z,s) ds$ and consider the functional $\mu_\lambda : W^{1,p}(\Omega) \rightarrow\RR\,\, (\lambda \in (0,\lambda_0])$ defined by
	$$\mu_\lambda(u) = \frac{1}{p} \gamma_p (u) + \frac{1}{q} ||D u||_q^q - \int_{\Omega} \theta_\lambda(z,u)dz\ \mbox{for all }u\in W^{1,p}(\Omega).$$

	Since $0\leq\overline{u}^{-\gamma} \leq v^{-\gamma} \in L^{\infty}(\Omega)$, we see that $\mu_\lambda \in C^1(W^{1,p}(\Omega))$. Also, it is clear from \eqref{eq18} and \eqref{eq1}, that $\mu_\lambda(\cdot)$ is coercive. In addition, it is sequentially weakly lower semicontinuous. So, we can find $u_\lambda \in W^{1,p}(\Omega)$ such that
	\begin{eqnarray}
		&&\mu_\lambda(u_\lambda) = \inf \left\{ \mu_\lambda(u): u\in W^{1,p}(\Omega) \right\}, \nonumber \\
		&&\Rightarrow \mu_\lambda^{'} (u_\lambda) = 0, \nonumber \\
		&\Rightarrow & \langle A_p(u_\lambda),h\rangle + \langle A_q(u_\lambda),h \rangle + \int_{\Omega}\xi(z)| u_\lambda|^{p-2} \nonumber u_\lambda hdz + \int_{\partial\Omega} \beta(z) |u_\lambda|^{p-2} u_\lambda hd\sigma \nonumber \\
		&=& \int_{\Omega} \vartheta_\lambda (z,u_\lambda) hdz\ \mbox{for all } h\in W^{1,p}(\Omega). \label{eq19}
	\end{eqnarray}

	In \eqref{eq19} first we choose $h=(u_\lambda -\overline{u})^+ \in W^{1,p}(\Omega)$. Then
\begin{eqnarray*}
	&&\langle A_p(u_\lambda),(u_\lambda - \overline{u})^+\rangle + \langle A_q(u_\lambda),(u_\lambda-\overline{u})^+\rangle + \int_{\Omega}\xi (z) u_\lambda^{p+} (u_\lambda -\overline{u})^{+} dz + \\
	&& \int_{\partial\Omega} \beta(z) u_\lambda^{p-1} (u_\lambda - \overline{u}) d\sigma \\
	&&=\int_{\Omega} [\overline{u}^{-\gamma} + \lambda f(z,\overline{u})](u_\lambda -\overline{u})^+ dz\ \mbox{(see \eqref{eq18})) } \\
	&&\leq \int_{\Omega} [\overline{u}^{-\gamma} +1](u_\lambda -\overline{u})^+ dz\ \mbox{(see \eqref{17})} \\
	&&\leq \int_{\Omega} [v^{-\gamma} + 1](u_\lambda -\overline{u})^+ dz\ \mbox{(since } v\leq\overline{u}) \\
	&&= \langle A_p(\overline{u}),(u_\lambda -\overline{u})^+ \rangle + \langle A_q(\overline{u}),(u_\lambda-\overline{u})^+ \rangle + \int_{\Omega} \xi (z) \overline{u}^{p-1} (u_\lambda -\overline{u})^+ dz \\
	&&+ \int_{\partial\Omega} \beta(z) \overline{u}^{p-1} (u_\lambda -\overline{u})^+ d\sigma\ \mbox{(see Proposition \ref{prop9}),} \\
	&&\Rightarrow u_\lambda \leq \overline{u}.
\end{eqnarray*}

	Next, in \eqref{eq19} we choose $h=(v-u_\lambda)^+ \in W^{1,p}(\Omega).$ Then
\begin{eqnarray*}
	&&\langle A_p(u_\lambda),(v-u_\lambda)^+\rangle + \langle A_q(u_\lambda),(v-u_\lambda)^+ \rangle + \int_{\Omega} \xi (z) |u_\lambda|^{p-2} u_\lambda (v-u_\lambda)^+ dz + \\
	&& \int_{\partial\Omega} \beta(z) |u_\lambda|^{p-2} u_\lambda(v-u_\lambda)^+ d\sigma \\
	&&=\int_{\Omega} [v^{-\gamma} + \lambda f(z,v) ] (v-u_\lambda)^+ dz \mbox{(see \eqref{eq18})} \\
	&&\geq \int_{\Omega} v^{-\gamma} (v-u_\lambda)^+ dz \mbox{(since } f\geq 0) \\
	&&=\,\,\langle A_p(v),(v-u_\lambda)^+\rangle + \langle A_q(v),(v-u_\lambda)^+\rangle + \int_{\lambda} \xi (z) v^{p-1} (v-u_\lambda)^+ dz \\
	&&+ \int_{\partial\Omega}\beta (z) v^{p-1} (v-u_\lambda)^+ d\sigma\ \mbox{(see Proposition \ref{prop8}),} \\
	&&\Rightarrow v \leq u_\lambda.
\end{eqnarray*}

	So, we have proved that
	\begin{equation}\label{eq20}
		u_\lambda \in [v,\overline{u}].
	\end{equation}

	From \eqref{eq18}, \eqref{eq19}, \eqref{eq20} it follows that
	\begin{equation}\label{eq21}
		\left\{
			\begin{array}{ll}
			-\Delta_p u_\lambda(z) -\Delta_q u_\lambda(z) + \xi(z) u_\lambda (z)^{p-1} = u_\lambda (z)^{-\gamma} + \lambda f (z,u_\lambda(z)) \\ \mbox{for almost all } z\in \Omega, \\
			\frac{\partial u_\lambda}{\partial n_{pq}} + \beta(z) u_\lambda^{p-1} = 0\ \mbox{on}\ \partial\Omega,\ \mbox{(see \cite{7})}.
			\end{array}
		\right\}
	\end{equation}

	From \eqref{eq21} and Proposition \ref{prop7} of Papageorgiou \& R\u adulescu \cite{8}, we have that $u_\lambda\in L^{\infty}(\Omega)$. So, the nonlinear regularity theory of Lieberman \cite{6} implies that $u_\lambda \in D_+$ (see \eqref{eq20}). Therefore we have proved that
	$$
		(0,\lambda_0] \leq {\mathcal L} \neq \emptyset\ \mbox{and}\ S_\lambda \subseteq D_+.
	$$
The proof is now complete.
\end{proof}

	Next, we establish a lower bound for the elements of $S_\lambda$.

\begin{prop}\label{prop11}
If hypotheses $H(\xi),H(\beta),H_0,H(f)$ hold, $\lambda\in {\mathcal L}$ and $u\in S_\lambda$, then $v\leq u$.
\end{prop}

\begin{proof}
	From Proposition \ref{prop10} we know that $u\in D_+$. Then Proposition \ref{prop7} implies that for $\eta >0$ small we have $\tilde{u}_\eta \leq u$. So, we can define the following Carath\'eodory function
	\begin{equation}\label{eq22}
	e(z,x) =
	\left\{
		\begin{array}{ll}
				\tilde{u}_\eta (z)^{-\gamma} & \mbox{if}\ x<\tilde{u}_\eta (z) \\
				x^{-\gamma} & \mbox{if}\ \tilde{u}_\eta(z) \leq x \leq u(z) \\
				u(z)^{-\gamma} & \mbox{if}\ u(z) <x.
		\end{array}
	\right.
	\end{equation}

	We set $E(z,x)= \int_{0}^x e(z,s) ds$ and consider the functional $d:W^{1,p}(\Omega) \rightarrow\RR$ defined by
	$$
		d(u) = \frac{1}{p} \gamma_p(u) + \frac{1}{q} ||D u||_q^q - \int_\Omega E(z,u) dz \mbox{for all } u\in W^{1,p}(\Omega).
	$$

	As before, we have $d\in C^1(W^{1,p}(\Omega))$. Also, $d(\cdot)$ is coercive (see \eqref{eq22}) and weakly lower semicontinuous. Hence we can find $\hat{v} \in W^{1,p}(\Omega) $ such that
	\begin{equation}\label{eq23}
	\begin{array}{ll}
		&\displaystyle d(\hat{u})= \inf \{d(u): u\in W^{1,p}(\Omega)\}, \\
		\Rightarrow &\displaystyle d'(\hat{v}) = 0,\\
		\Rightarrow &\displaystyle \langle A_p(\hat{v}),h\rangle + \langle A_q(\hat{v}),h\rangle + \int_{\Omega} \xi (z) |\hat{v}|^{p-2} \hat{v} hdz+\int_{\partial\Omega} \beta(z) |\hat{v}|^{p-2} \hat{v} hd\sigma = \\
		& \displaystyle \int_{\Omega} e(z,\hat{v}) hdz\ \mbox{for all}\ h\in W_{1,p}(\Omega).\end{array}
	\end{equation}

	In \eqref{eq23} first we choose $h=(\hat{v}-u)^+ \in W^{1,p}(\Omega).$ Exploiting the fact that $u\in S_\lambda$ and recalling that $f\geq 0$, we obtain $\hat{v} \leq u$. Next in \eqref{eq23} we test with $h=(\tilde{u}_\eta -v)^+\in W^{1,p}(\Omega)$. Using \eqref{eq22}, \eqref{eq9} and Proposition \ref{prop7}, we obtain $\tilde{u}_\eta \leq \hat{v}$. Therefore
	\begin{equation}\label{eq24}
		\hat{v} \in [\tilde{u}_\eta,u].
	\end{equation}

	From \eqref{eq22}, \eqref{eq23}, \eqref{eq24} and Proposition \ref{prop8}, we conclude that
	$$
		\begin{array}{ll}
			& \hat{v} = v, \\
			\Rightarrow & v \leq u\ \mbox{for all } u\in S_\lambda.
		\end{array}
	$$
The proof is now complete.
\end{proof}

	Now we can deduce a structural property of $\mathcal{L}$.
\begin{prop}\label{prop12}
If hypotheses $H(\xi),H(\beta),H_0,H(f)$ hold, $\lambda \in \mathcal{L}$, $0<\mu < \lambda$ and $u_\lambda \in S_\lambda \subseteq D_+,$ then $\mu \in \mathcal{L}$ and we can find $u_\mu \in S_\mu \subseteq D_+$ such that $u_\lambda -u_\mu \in\ {\rm int}\,\hat{C}_+$.
\end{prop}

\begin{proof}
From Proposition \ref{prop11} we know that $v \leq u_\lambda.$ Then we can define the following Carath\'eodory function
	\begin{equation}\label{eq25}
		\hat{k}_\mu(z,x) =
	\left\{
	 	\begin{array}{ll}
			v(z)^{-\gamma} + \mu f(z,v(z)) & \mbox{if}\ x< v(z)\\
			x^{-\gamma} + \mu f (z,x) & \mbox{if}\ v(z) \leq x \leq u_\lambda (z) \\
			u_\lambda(z)^{-\gamma} + \mu f(z,u_\lambda(z)) & \mbox{if}\ u_\lambda(z) < x.
	 	\end{array}
	\right.
	\end{equation}

	We set $\hat{K}_\mu(z,x) = \int_{0}^{x} \hat{k}_\mu (z,s) ds$ and consider the $C^1$-functional $\hat{\Psi}_\mu:W^{1,p}(\Omega) \rightarrow\RR$ defined by
	$$
	\hat{\Psi}_\mu (u) = \frac{1}{p} \gamma_p(u) + \frac{1}{q} || D u ||_q^q - \int_{\Omega} \hat{K}_\mu (z,u) dz\ \mbox{for all}\ u \in W^{1,p}(\Omega).
	$$

	Evidently, $\hat{\Psi}_\mu(\cdot)$ is coercive (see \eqref{eq25}) and sequentially weakly lower semicontinuous. So, we can find $u_\mu \in W^{1,p}(\Omega)$ such that
	\begin{eqnarray}
		&&\hat{\Psi}_\mu(u_\mu) = \inf \left\{ \hat{\Psi}_\mu(u): u\in W^{1,p}(\Omega) \right\}, \nonumber \\
		&\Rightarrow & \hat{\Psi}_\mu' (u_\mu) = 0, \nonumber \\
		&\Rightarrow & \langle A_p(u_\mu),h\rangle + \langle A_q(u_\mu),h\rangle + \int_{\Omega} \xi(z) |u_\mu|^{p-2} u_\mu hdz + \int_{\partial\Omega} \beta (z) |u_\mu|^{p-2} u_\mu hd\sigma \nonumber \\
		&=& \int_{\Omega} \hat{k}_\mu (z,u\mu) hdz\ \mbox{for all } h \in W^{1,p}(\Omega). \label{eq26}
	\end{eqnarray}

	In \eqref{eq26} first we choose $h = (u_\mu - u_\lambda)^+ \in W^{1,p}(\Omega).$ Using \eqref{eq25}, the fact that $\mu < \lambda$ and that $f \geq 0$ and recalling that $u_\lambda \in S_\lambda$, we conclude that $u_\mu \leq u_\lambda$. Next, in \eqref{eq26} we choose $h=(v-u_\mu)^+ \in W^{1,p}(\Omega).$ From \eqref{eq25}, the fact that $f \geq 0$ and Proposition \ref{prop8}, we infer that $v \leq u_\mu$. Therefore we have proved that
	\begin{equation}\label{eq27}
		u_\mu \in [v,u_\lambda].
	\end{equation}

	From \eqref{eq25}, \eqref{eq26}, \eqref{eq27} it follows that
	$$
	u_\mu \in S_\mu \subseteq D_+ \mbox{(see Proposition \ref{prop10}).}
	$$

	Let $\rho= ||u_\lambda||_{\infty} $ and let $\hat{\xi}_{\rho}^\lambda > 0$ be as postulated by hypothesis $H(f)(v)$. We have
	\begin{eqnarray}
		&&-\Delta_p u_\lambda (z) - \Delta_q u_\mu (z) + \left[\xi(z) + \hat\xi_\rho^{\lambda}\right] u_\mu(z)^{p-1} - u_\mu(z)^{-\gamma} \nonumber \\
		&=& \mu f(z,u_\mu(z)) + \hat\xi_\rho^{\lambda} u_\mu(z)^{p-1} \nonumber \\
		&=&\lambda f(z,u_\mu(z)) + \hat\xi_\rho^{\lambda} u_\mu (z)^{p-1} - (\lambda - \mu) f(z,u_\mu (z)) \nonumber \\
		&<& \lambda f (z,u_\mu(z)) + \hat\xi_\rho^{\lambda} u_\lambda (z)^{p-1}\ \mbox{(recall that } \lambda > \mu ) \nonumber \\
		&\leq& \lambda f (z,u_\lambda(z)) + \hat\xi_\rho^{\lambda} u_\lambda(z)^{p-1}\ \mbox{(see \eqref{eq27} and hypothesis } H(f)(v)) \nonumber \\
		&=& -\Delta_p u_\lambda (z) - \Delta_q u_\lambda (z) + \left[ \xi(z) + \hat\xi_\rho^{\lambda} \right] u_\lambda(z)^{p-1} - u_\lambda(z)^{-\lambda}\ \mbox{for almost all}\ z \in \Omega \label{eq28} \\
		&&\mbox{(recall that }u_\lambda \in S_\lambda). \nonumber
	\end{eqnarray}

	We know that
	$$
	0 \leq u_\mu^{-\gamma},\, u_\lambda^{-\gamma} \leq v^{-\gamma} \in L^{\infty}(\Omega).
	$$

	Also, from hypothesis $H(f)(iv)$ and since $u_\mu \in D_+,$ we have
	$$
	0 < c_8 \leq (\lambda - \mu)f(z,u_\mu(z)) \ \mbox{for almost all } z\in \Omega.
	$$

	Invoking Proposition \ref{prop4}, from \eqref{eq28} we conclude that
	$$
	u_\lambda - u_\mu \in \mbox{int}\, \hat{C}_+.
	$$
The proof is now complete.
\end{proof}

\begin{prop}\label{prop13}
If hypotheses $H(\xi),H(\beta),H_0,H(f)$ hold, then $\lambda^* < + \infty$.
\end{prop}

\begin{proof}
On account of hypotheses $H(f)(i) \rightarrow (iv)$, we can find $\lambda_0 >0$ big such that
	\begin{equation}\label{eq29}
		x^{-\gamma} + \lambda_0 f(z,x) \geq x^{p-1}\ \mbox{for almost all } z\in \Omega  \ \mbox{and all } x \geq 0.
	\end{equation}

	Let $\lambda > \lambda_0$ and suppose that $\lambda \in \mathcal{L}$. Then we can find $u_\lambda \in S_\lambda \subseteq D_+$ (see Proposition \ref{prop10}). Then $m_\lambda = \min_{\overline{\Omega}} u_\lambda >0$. For $\delta\in (0,1)$ we set $m_\lambda^{\delta} = m_\lambda + \delta$ and for $\rho = ||u_\lambda||_{\infty}$ let $\hat\xi_\rho^{\lambda} > 0$ be as postulated by hypothesis $H(f)(v)$. We have

	\begin{eqnarray}
	&& -\Delta_p m_\lambda^\delta - \Delta_q m_\lambda^{\delta} + [\xi(z) + \hat\xi_\rho](m_\lambda^{\delta})^{p-1} - (m_\lambda^{\delta})^{-\gamma} \nonumber \\
	& =&[\xi(z) + \hat\xi_\rho^{\lambda}]m_{\lambda}^{p-1} - m_\lambda^{-\gamma} + \chi(\delta)\  \mbox{with } \chi(\delta) \rightarrow 0^{+} \mbox{as } \delta \rightarrow 0^+ \nonumber \\
	& <& \xi(z) m_\lambda^{p-1} + (1+ \hat\xi_\rho^{\lambda}) m_\lambda^{p-1} - m_\lambda^{-\gamma} + \chi(\delta) \nonumber\\
	& \leq& \lambda_0 f(z,m_\lambda) + [\xi(z) + \hat\xi_\rho^{\lambda}] m_\lambda^{p-1} + \chi(\delta)\ \mbox{(see \eqref{eq29})} \nonumber \\
	& \leq& \lambda_0 f(z,u_\lambda) + [\xi(z) + \hat\xi_\rho^{\lambda}] u_\lambda^{p-1} + \chi(\delta) \mbox { (see hypothesis } H(f)(v)) \nonumber \\
	& =& \lambda f(z,u_\lambda) + [\xi(z) + \hat\xi_\rho^{\lambda}] u_\lambda^{p-1} -(\lambda -\lambda_0) f(z,u_\lambda) + \chi(\delta) \nonumber\\
	& = &\lambda f(z,u_\lambda) + [\xi(z) + \hat\xi_\rho^{\lambda}] u_\lambda^{p-1} \mbox { for } \delta \in (0,1) \mbox { small } \nonumber \\
	&& \mbox{(recall that}\ u_\lambda \in D_+ \mbox{and see } H(f)(iv)) \nonumber \\
	& =& -\Delta_p u_\lambda - \Delta_q u_\lambda + [\xi(z) + \hat\xi_\rho^{\lambda}] u_\lambda^{p-1} - u_\lambda^{-\gamma}. \label{eq30}
	\end{eqnarray}
	
	Since $(\lambda - \lambda_0)f(z,u_\lambda) - \chi(\delta) \geq c_9 > 0$ for almost all $z \in \Omega$ and for $\delta\in(0,1)$ small (just recall that $u_\lambda \in D_+$ and use hypothesis H(f)(iv), invoking Proposition \ref{prop4}, from \eqref{eq30} we infer that
	$$
	u_\lambda - m_\lambda^{\delta} \in {\rm int}\, \hat{C}_+ \ \mbox{for all }\delta \in (0,1)\ \mbox{small enough}.
	$$

	However,
	this contradicts the definition of $m_\lambda$.
	It follows that $\lambda \notin {\mathcal L}$ and so $\lambda^* \leq \lambda_0 < +\infty$.
\end{proof}

Therefore we have
$$
(0,\lambda^*) \subseteq {\mathcal L} \subseteq (0,\lambda^*].
$$

\begin{prop}\label{prop14}
If hypotheses $H(\xi),H(\beta),H_0,H(f)$ hold and $\lambda \in (0,\lambda^*)$, then problem \eqref{eqp} has at least two positive solutions
$$
u_0,\ \hat{u} \in D_+,\ u_0 \neq \hat{u}.
$$
\end{prop}

\begin{proof}
Let $0<\mu<\lambda<\eta <\lambda^*$. According to Proposition \ref{prop12}, we can find $u_\eta \in S_\eta \subseteq D_+$, $u_0\in S_\lambda \subseteq D_+$ and $u_\mu \in S_\mu \subseteq D_+ $ such that

	\begin{equation}\label{eq31}
	\begin{aligned}
	& u_\eta - u_0 \in {\rm int}\, \hat{C}_+\  \mbox{and } u_0 -u_\mu \in {\rm int}\, \hat{C}_+,\\
	& \Rightarrow u_0 \in {\rm int}_{C^1 (\hat{\Omega})} [u_\mu,u_\eta].
	\end{aligned}
	\end{equation}

	We introduce the following Carath\'eodory function
	\begin{equation}\label{eq32}
	\tilde\tau_\lambda (z,x) = \left\{
	 \begin{array}{ll}
				u_\mu(z)^{-\gamma} + \lambda f (z,u_\mu(z)) & \mbox{if } x < u_\mu(z) \\
				x^{-\gamma} + \lambda f (z,x) & \mbox{if } u_\mu(z) \leq x \leq u_\eta(z) \\
				u_\eta(z)^{-\gamma} + \lambda f (z,u_\eta(z)) & \mbox{if } u_\eta(z) <x.
	 \end{array}
	\right.
	\end{equation}

	Set $\tilde{T}_\lambda(z,x) = \int_{0}^{x} \tilde{\tau}_\lambda (z,s)ds$ and consider the $C^1$-functional $\tilde{\Psi}_\lambda: W^{1,p}(\Omega) \rightarrow\RR$ defined by
	$$
	\tilde{\Psi}_\lambda(u) = \frac{1}{p}\gamma_p(u) + \frac{1}{q} ||D u||_q^q - \int_{\lambda}\tilde{T}_\lambda (z,u)dz\ \mbox{for all } u\in W^{1,p}(\Omega).
	$$

	Using \eqref{eq32} and the nonlinear regularity theory, we can easily check that
	\begin{equation}
	K_{\tilde{\Psi}_\lambda} \subseteq [u_\mu,u_\eta] \cap D_+.
	\label{eq33}
	\end{equation}

	Also, consider the Carath\'eodory function
	\begin{equation}
	\tau_\lambda^* (z,x) =  \left\{
	 \begin{array}{ll}
				u_\mu(z)^{-\gamma} + \lambda f (z,u_\mu(z)) & \mbox{if } x \leq u_\mu (z)\\
				x^{-\gamma} + \lambda f(z,x) & \mbox{if } u_\mu (z) < x.
	 \end{array}
	\right.
	\label{eq34}
	\end{equation}

	We set $T_\lambda^* (z,x) = \int_{0}^{x} \tau_\lambda^* (z,s)ds$ and consider the $C^1$-functional $\Psi_\lambda^*: W^{1,p}(\Omega) \rightarrow\RR $ defined by
	$$ \Psi_\lambda^*(u) = \frac{1}{p} \gamma_p(u) + \frac{1}{q}||D u||_q^q - \int_{\Omega} T_\lambda^* (z,u) dz\ \mbox{for all } u\in W^{1,p}(\Omega). $$

	For this functional using \eqref{eq34}, we show that
	\begin{equation}
	K_{\Psi_\lambda^*} \subseteq [u_\mu) \cap D_+.
	\label{eq35}
	\end{equation}

	From \eqref{eq32} and \eqref{eq34} we see that
	\begin{equation}
	\tilde{\Psi}_\lambda \Big|_{[u_\mu,u_\eta]} = \Psi_\lambda^{*} \Big|_{[u_\mu,u_\eta]} \,\,\mbox{and }\,\, \tilde{\Psi}_\lambda^{'}\Big|_{[u_\mu,u_\eta]} = (\Psi_\lambda^*)' \Big|_{[u_\mu,u_\lambda]}.
	\label{eq36}
	\end{equation}

	From \eqref{eq33}, \eqref{eq35}, \eqref{eq36}, it follows that without any loss of generality, we may assume that
	\begin{equation}
	K_{\Psi_\lambda^*} \cap [u_\mu,u_\eta] = \{u_0\}.
	\label{eq37}
	\end{equation}

	Otherwise it is clear from \eqref{eq34} and \eqref{eq35} that we already have a second positive smooth solution for problem \eqref{eqp} and so we are done.

	Note that $\tilde{\Psi_\lambda}(\cdot)$ is coercive (see \eqref{eq32}). Also, it is sequentially weakly lower semicontinuous. So, we can find $\hat{u}_0 \in W^{1,p}(\Omega)$ such that
	\begin{equation}
	\begin{aligned}
	&\tilde{\Psi}_\lambda (\hat{u}_0) = \inf  \left\{ \tilde{\Psi}_\lambda (u): u\in W^{1,p}(\Omega)\right\},\\
	&\Rightarrow \hat{u}_0 \in K_{\tilde{\Psi}_\lambda}, \\
	&\Rightarrow \hat{u}_0 \in K_{\Psi_\lambda^*} \cap [u_\mu,u_\eta]\ \mbox{(see \eqref{eq33},\eqref{eq36}) }, \\
	&\Rightarrow \hat{u}_0 = u_0 \in D_+\ \mbox{(see \eqref{eq37}}), \\
	&\Rightarrow u_0\ \mbox{is a local}\ C^1(\overline{\Omega})\mbox{-minimizer of } \Psi_\lambda^*\ \mbox{(see \eqref{eq31})},\\
	&\Rightarrow u_0\ \mbox{is a local } W^{1,p}(\Omega)\mbox{-minimizer of }\Psi_\lambda^*\ \mbox{(see Proposition \ref{prop5}).}
	\end{aligned}
	\label{eq38}
	\end{equation}

	We assume that $K_{\Psi_\lambda^*}$ is finite. Otherwise on account of \eqref{eq34} and \eqref{eq35} we see that we already have an infinity of positive smooth solutions for problem \eqref{eqp} and so we are done. Then \eqref{eq38} implies that we can find $\rho\in(0,1)$ small. such that

	\begin{equation}
	\begin{aligned}
	&\Psi_\lambda^*(u_0)< \inf  \left\{ \Psi_\lambda^*(u): ||u-u_0|| = \rho \right\} = m_\lambda^* \\
	&\mbox{(see Papageorgiou, R\u adulescu \& Repov\v{s} \cite[Theorem 5.7.6, p. 367]{12}).}
	\end{aligned}
	\label{eq39}
	\end{equation}

	On account of hypothesis $H(f)(ii)$ we have
	\begin{equation}
	\Psi_\lambda^* (t\hat{u}_1 (p)) \rightarrow -\infty\ \mbox{as } t\rightarrow +\infty.
	\label{eq40}
	\end{equation}

	\begin{claim}
		$\Psi_\lambda^*(\cdot)$ satisfies the C - condition.
	\end{claim}
	Let $\{u_n\}_{n \geq 1} \subseteq \mbox{W}^{1,p}(\Omega)$ be a sequence such that
	\begin{equation} |\Psi_\lambda^* (u_n) |\leq c_{10}\ \mbox{for some }  c_{10} > 0   \ \mbox{and all } n \in \NN,
	\label{eq41}
	\end{equation}
	\begin{equation} (1 + ||u_n|| ) (\Psi_\lambda^*)' (u_n) \rightarrow 0\ \mbox{in W }^{1,p}(\Omega)^*.
	\label{eq42}
	\end{equation}

	From \eqref{eq42} we have
	\begin{equation}
	\begin{aligned}
	&| \langle A_p(u_n),h\rangle + \langle A_q(u_n),h\rangle + \int_{\Omega} \xi(z) |u_n|^{p-2}u_n h \, dz + \int_{\partial\Omega} \beta(z) |u_n|^{p-2} u_n hd\sigma\\
	& - \int_{\Omega} \tau_\lambda^*(z, u_n) h \,dz | \leq \frac{\epsilon_n ||h||}{1 + ||u_n||}\ \mbox{for all } h \in W^{1,p},\ \mbox{with } \epsilon_n \rightarrow 0^+.
	\end{aligned}
	\label{eq43}
	\end{equation}

	Choosing  $h= -u_n^{-} \in W^{1,p}(\Omega)$, we obtain
	\begin{eqnarray}
		&&\gamma_p(u_n^{-}) + ||D u_n^{-} ||_q^q \leq c_{11} ||u_n^{-}||\ \mbox{for some } c_{11} > 0  \ \mbox{and all } n \in \NN\ \mbox{(see \eqref{eq34})} \nonumber \\
		&\Rightarrow &\{u_n^{-} \}_{n \geq 1} \subseteq W^{1,p}(\Omega) \  \mbox{is bounded } \mbox{(see \eqref{eq1} and recall that }1<p). \label{eq44}
	\end{eqnarray}
	
	Next in \eqref{eq43} we choose $ h = u_n^+ \in W^{1,p}(\Omega)$. Then
	\begin{equation}
	\begin{aligned}
	&-\gamma_p (u_n^{+}) - || Du_n^+ ||_q^q + \int_{\Omega} \tau_\lambda^* (z,u_n) u_n^+ dz \leq \epsilon_n\ \mbox{for all } n \in \NN,\\
	&\Rightarrow -\gamma_p(u_n^+) - ||Du_n^+||_q^q + \int_{\{ u_n \leq u_{\mu} \}} [u_\mu^{-\gamma} + \lambda f(z,u_\mu)] u_n^{+}dz \\
	&+ \int_{\{ u_\mu < u_n \}} [u_n^{-\gamma}+\lambda f(z,u_n)]u_n^+ dz  \, \leq \, \epsilon_n\ \mbox{for all } n\in \NN\ \mbox{(see \eqref{eq34})}
	\end{aligned}
	\label{eq45}
	\end{equation}

	On the other hand from \eqref{eq41} and \eqref{eq44}, we have
	$$ \gamma_p(u_n^+) + \frac{p}{q} || Du_n^+ ||_q^q - \int_{ \{u_n\leq u_\mu\}} p[u_\mu^{-\gamma} + \lambda f(z,u_p) ] u_n^+ \, dz $$
	\begin{equation*}
	\begin{aligned}
	 &- \int_{\{u_\mu < u_n \}} \bigg[\frac{p}{1-\gamma} (u_n^{1-\gamma} - u_\mu^{1-\gamma}) + p(\lambda F (z,u_n) - \lambda F(z,u_\mu)) \bigg] dz \leq \epsilon_n \\
	&\quad \mbox{for all } n \in \NN\ (\mbox{see}\ \eqref{eq34}).
	\end{aligned}
	\end{equation*}
	\begin{equation}
	\begin{aligned}
	 &\Rightarrow \gamma_p(u_n^+) + \frac{p}{q} ||D u_n^+ ||_p^p - \int_{ \{ u_n \leq u_\mu \}} p [u_\mu^{-\gamma} + \lambda f(z,u_\mu)] u_n^+ dz \\
	&- \int_{ \{ u_p < u_n \}} \bigg[ \frac{p}{1-\gamma} u_n^{1-\gamma} + \lambda p F(z,u_n)] dz \leq c_{12}\
	 \mbox{for some } c_{12} > 0 \ \mbox{and all } n \in \NN.
	\end{aligned}
	\label{eq46}
	\end{equation}

	We add \eqref{eq45} and \eqref{eq46}. Since $p > q$, we obtain
	\begin{eqnarray}
		&&\lambda \int_{ \{ u_\mu < u_n \} } [f(z,u_n)u_n^+ - pF(z,u_n) ] dz \leq \,\, (p-1) \int_{ \{ u_n \leq u_\mu \}} [u_\mu^{-\gamma} + \lambda f(z,u_\mu)] u_n^+ dz \nonumber \\
		&&+\left(\frac{p}{1-\gamma} - 1 \right) \int_{\{ u_\mu < u_n \}} u_n^{1-\gamma}dz \nonumber \\
		&\Rightarrow& \lambda \int_{\Omega} [f(z,u_n^+)u_n^+ - p F (z,u_n^+)] dz \,\, \leq \,\, c_{13} \,\left[|| u_n^+ ||_{1} + 1\right] \label{eq47}\\
		&& \mbox{for some } c_{13}>0, \mbox{all }  n \in \NN. \nonumber
	\end{eqnarray}

	On account of hypotheses $ H(f)(i),(iii)$ we can find $\hat{\beta}_1 \in (0,\hat{\beta}_0) $ and $c_{14} > 0$ such that
	\begin{equation}
	\hat{\beta}_1 x^\tau - c_{14} \leq f(z,x) - p F(z,x)\ \mbox{for almost all } z \in \Omega \ \mbox{and all } x \geq 0.
	\label{eq48}
	\end{equation}

	Using \eqref{eq48} in \eqref{eq47}, we obtain
	$$ ||u_n^+||_\tau^\tau \leq c_{15} \left[ ||u_n^+||_\tau + 1 \right] \mbox{for some } c_{15} > 0 \ \mbox{and all } n\in \NN, $$
	\begin{equation}
	\Rightarrow \{u_n^+\}_{n \geq 1} \leq L^\tau (\Omega)\ \mbox{is bounded}.
	\label{eq49}
	\end{equation}

	First assume  $N\neq p$. From hypothesis $H(f) (iii)$ it is clear that we may assume without any loss of generality that $ \tau < r < p^*. $ Let $t\in(0,1)$ be such that
	$$ \frac{1}{r} = \frac{1-t}{\tau} + \frac{t}{p*}\,. $$

	Then from the interpolation inequality (see Papageorgiou \& Winkert \cite[Proposition 2.3.17, p. 116]{15}), we have
	\begin{eqnarray}
		&&|| u_n^+ ||_r \leq || u_n^+ ||_\tau^{1-t} ||u_n^+||_{p^*}^{t}, \nonumber \\
		& \Rightarrow & ||u_n^+||_r^r \leq c_{16} ||u_n^+||^{tr} \mbox{for some } c_{16} > 0 \ \mbox{and all } n \in \NN\ \mbox{(see \eqref{eq49})}. \label{eq50}
	\end{eqnarray}

	From hypothesis $H(f)(i)$ we have
	\begin{equation}
	f(z,x) x \leq c_{17} [1+ x^r]\ \mbox{for all }z \in \Omega ,\ \mbox{all } x \geq 0 \ \mbox{and some } c_{17} > 0.
	\label{eq51}
	\end{equation}

	From \eqref{eq43} with $h=u_n^+ \in W^{1,p} (\Omega)$, we obtain
\begin{eqnarray}
	&& \gamma_p (u_n^+) + || D u_n^+ ||_q^q - \int_{\Omega} \tau_\lambda^* (z,u_n) u_n^+ dz \leq \epsilon_n\ \mbox{for all 	} n\in \NN,\nonumber\\
	&\Rightarrow& \gamma_p (u_n^+) + || D u_n^+ ||_q^q \leq \int_{\Omega} [(u_n^+)^{1-\gamma} + f(z,u_n^+) u_n^+] dz + c_{18} \nonumber\\
	&&\mbox{for some } c_{18} > 0
	\ \mbox{and all } n \in \NN\ \mbox{(see \eqref{eq34}) } \nonumber\\
	&\leq& c_{19} \left[ 1 + || u_n^+||_r^r \right] \mbox{for some } c_{19} > 0
	\ \mbox{and all } n\in \NN\ \mbox{(see \eqref{eq51})}\nonumber\\
	&\leq& c_{20} [1 + ||u_n^+||^{tr}]\ \mbox{for some } c_{20} > 0
	 \ \mbox{and all } n\in \NN\ \mbox { (see \eqref{eq50})}. \label{eq52}
\end{eqnarray}

	The hypothesis on $\tau$ (see $H(f)(iii)) $ implies that $tr < p.$ So, from \eqref{eq52} we infer that
	$$ \{ u_n^+ \}_{n \geq 1} \subseteq W^{1,p}(\Omega)\ \mbox{is bounded}, $$
	\begin{equation}
	\Rightarrow \{ u_n \}_{n \geq 1} \subseteq W^{1,p}(\Omega)\ \mbox{is bounded (see \eqref{eq44})}.
	\label{eq53}
	\end{equation}

	If $N = p,$ then $p^* = + \infty $ and from the Sobolev embedding theorem, we know that $W^{1,p}(\Omega) \hookrightarrow L^s(\Omega)$ for all $1\leq s < \infty$. Then in order for the previous argument to work, we replace $p^* = + \infty$ by $s > r > \tau$ and let $t\in (0,1)$ as before such that
	$$ \frac{1}{r} = \frac{1-t}{\tau} + \frac{t}{s}, $$
	$$ \Rightarrow tr = \frac{s(r-\tau)}{s-\tau}. $$

	Note that $ \frac{s(r-\tau)}{s-\tau} \rightarrow r -\tau $ as $s \rightarrow + \infty$. But $r-\tau <p$ (see hypothesis H(f)(iii)). We choose $s>r$ big so that $tr<p$. Then again we have \eqref{eq53}.

	Because of \eqref{eq53} and by passing to a subsequence if neccesary, we may assume that
	\begin{equation}
	u_n \xrightarrow{w} u\ \mbox{in } W^{1,p}(\Omega)\ \mbox{and } u_n \rightarrow u\ \mbox{in } L^r (\Omega)\ \mbox{and  }L^p(\partial\Omega).
	\label{eq54}
	\end{equation}

	In \eqref{eq43} we choose $h = u_n - u \in W^{1,p}(\Omega)$, pass to the limit as $n \rightarrow \infty$ and use \eqref{eq54}. Then
	$$ \lim_{n\rightarrow\infty} \left[\langle A_p (u_n),u_n -u\rangle + \langle A_q(u_n),u_n -u\rangle\right] = 0, $$
	\begin{equation*}
	\begin{aligned}
	 &\Rightarrow \limsup_{n\rightarrow\infty} \left[\langle A_p(u_n),u_n -u\rangle + \langle A_q(u),u_n-u\rangle\right] \leq 0 \\
	 &\mbox{(since } A_q(\cdot)\ \mbox{is monotone}) \\
&\Rightarrow\limsup_{n\rightarrow\infty} \langle A_p(u_n),u_n -u\rangle \,\, \leq 0, \\
	 &\Rightarrow u_n \rightarrow u\ \mbox{in } W^{1,p}(\Omega)\ \mbox{(see Proposition \ref{prop1}).}
	\end{aligned}
	\end{equation*}

	Therefore $\Psi_\lambda^*(\cdot)$ satisfies the C-condition. This proves the claim.

	Then \eqref{eq39}, \eqref{eq40} and Claim permit the use of the mountain pass theorem.
	So, we can find $\hat{u}\in W^{1,p}(\Omega)$ such that
	$$ \hat{u} \in K_{\Psi_\lambda^*} \leq [u_\mu) \cap D_+ \mbox{(see \eqref{eq35}) } , m_\lambda^* \leq \Psi_\lambda^* (\hat{u})\ \mbox{(see \eqref{eq39})  }.$$

	Therefore $ \hat{u} \in D_+ $ is a second positive solution of problem \eqref{eqp}  $(\lambda \in (0,\lambda^*))$ distinct from $u_0 \in D_+$.
\end{proof}

Next, we examine what can be said in the critical parameter $\lambda^*$.

\begin{prop}\label{prop15}
If hypotheses $H(\xi),H(\beta),H_0,H(f)$ hold, then $\lambda^* \in {\mathcal L}$.
\end{prop}
\begin{proof}
Let $\{\lambda_n\}_{n \geq 1} \subseteq (0,\lambda^*) $ be such that $\lambda_n < \lambda^*.$ We can find $u_n\in S_{\lambda_n} \subseteq D_+$ for all $n \in \NN$.

	We consider the following Carath\'eodory function
	\begin{equation}
	\mu_n(z,x) = \left\{
	 \begin{array}{ll}
					v(z)^{-\gamma} + \lambda_n f(z,v(z)) & \mbox{if } x \leq v(z) \\
					x^{-\gamma} + \lambda_n f(z,x) &\mbox{if } v(z) < x.
	 \end{array}
	\right.
	\label{eq55}
	\end{equation}

	We set $M_n (z,x) = \int_{0}^x \mu_n (z,x) ds$ and consider the $C^1$-functional $j_n : W^{1,p}(\Omega) \rightarrow\RR$ defined by
	$$ j_n(u) = \frac{1}{p} \gamma_p(u) + \frac{1}{q} || D u ||_q^q  - \int_{\Omega}M_n(z,u) dz\ \mbox{for all } u\in W^{1,p}(\Omega).$$

	Also, we consider the following truncation of $\mu_n(z,\cdot)$
	\begin{equation}
	\hat\mu_n(z,x) = \left\{
	 \begin{array}{lr}
					\mu_n(z,x) & \mbox{if } x \leq u_{n + 1} (z) \\
					\mu_n(z,u_{n+1}(z)) & \mbox{if } u_{n+1} (z) < x
	 \end{array}
	\right.
	\label{eq56}
	\end{equation}
	(recall that $v\leq u_{n+1}$ for all $n\in \NN$, see Proposition \ref{prop11}). This is a Carath\'eodory function. We set $\hat{M}_n(z,x) = \int_{0}^x \hat{\mu}_n(z,s) ds$ and consider the $C^1$-functional $\hat{J}_n : W^{1,p}(\Omega) \rightarrow\RR$ defined by
	$$ \hat{J}_n (u) = \frac{1}{p} \gamma_p (u) + \frac{1}{q} || D u ||_q^q - \int_{\Omega} \hat{M}_n (z,u) dz\ \mbox{for all } u \in W^{1,p}(\Omega). $$

	From \eqref{eq55}, \eqref{eq56} and \eqref{eq1} , it is clear that $\hat{J}_n(\cdot)$ is coercive. Also, it is sequentially weakly lower semicontinuous. So, we can find $\hat{u}_n \in W^{1,p}(\Omega) $ such that
	\begin{equation}
	\hat{J}_n(\hat{u}_n) = \inf\left\{ \hat{J}_n(u): u \in W^{1,p}(\Omega) \right\}.
	\label{eq57}
	\end{equation}

	Then we have
	\begin{eqnarray}
		\hat{J}_n(\hat{u}_n) & \leq & \hat{J}_n(v) \nonumber \\
		& \leq & \frac{1}{p} \gamma_p(v) + \frac{1}{q} ||D v||_q^q - \frac{1}{1-\gamma} \int_{\Omega} v^{1-\gamma} dz \nonumber \\
		& &\mbox{(see \eqref{eq55}, \eqref{eq56} and recall that } f \geq 0 ) \nonumber \\
		& \leq & \langle A_p(v),v\rangle + \langle A_q(v),v\rangle - \int_{\Omega} v^{1-\gamma} dz = 0 \label{eq58} \\
		&& \mbox{(see Proposition \ref{prop8})}. \nonumber
	\end{eqnarray}

	From \eqref{eq57} we have
	\begin{equation}
	\hat{u}_n \in K_{\hat{J}_n} \subseteq [v,u_{n+1}] \cap D_+\ \mbox{for all } n\in \NN\ \mbox{(see \eqref{eq56})}.
	\label{eq59}
	\end{equation}

	Similarly, using \eqref{eq55} we obtain
	\begin{equation}
	K_{j_n} \subseteq [v) \cap D_+.
	\label{eq60}
	\end{equation}
	
	Note that
	$$ J_n |_{[v,u_{n+1}]} = \hat{J}_n |_{[v,u_{n+1}]}\ \mbox{and } J_n'|_{[v,u_{n+1}]} = \hat{J}_n'|_{[v,u_{n+1}]}\ \mbox{(see \eqref{eq55},\ \eqref{eq56}}).$$

	Then from \eqref{eq58}, \eqref{eq59}, \eqref{eq60}, we have
	\begin{equation}
	J_n(\hat{u}_n) \leq 0\ \mbox{for all } n\in \NN \\
	\label{eq61}
	\end{equation}
	\begin{equation}
	\begin{aligned}
	&\langle A_p(\hat{u}_n),h\rangle + \langle A_q(\hat{u}_n),h\rangle + \int_{\Omega} \xi(z)\hat{u}_n^{p-1} h dz + \int_{\partial\Omega} \beta(z)\hat{u}_n^{p-1} hd\sigma = \int_{\Omega} \mu_n(z,\hat{u}_n) h dz \\
	&\quad \mbox{for all } h \in W^{1,p}(\Omega),\ \mbox{all } n\in \NN.
	\end{aligned}
	\label{eq62}
	\end{equation}

	Using \eqref{eq61}, \eqref{eq62} and reasoning as in the Claim in the proof of Proposition \ref{prop14}, we show that
	$$ \{\hat{u}_n\}_{n\geq 1} \subseteq W^{1,p}(\Omega)\ \mbox{is bounded.} $$

	So, we may assume that
	\begin{equation}
	\hat{u}_n \stackrel{w}{\rightarrow} \hat{u}_*\ \mbox{in } W^{1,p}(\Omega)\ \mbox{and } \hat{u}_n \rightarrow \hat{u}_*\ \mbox{in } L^r(\Omega)\ \mbox{and } L^p(\partial\Omega).
	\label{eq63}
	\end{equation}

	In \eqref{eq62} we choose $h=\hat{u}_n - \hat{u}_* \in W^{1,p}(\Omega), $ pass to the limit as $n\rightarrow\infty$ and use \eqref{eq63}. Then as before (see the proof of Proposition \ref{prop14}), we obtain
	\begin{equation}
	\hat{u}_n \rightarrow \hat{u}_*\ \mbox{in } W^{1,p}(\Omega).
	\label{eq64}
	\end{equation}

	In \eqref{eq62} we pass to the limit as $n\rightarrow\infty$ and use \eqref{eq64}. Then
	$$ \langle A_p(\hat{u}_*),h\rangle + \langle A_q(\hat{u}_*),h\rangle + \int_{\Omega} \xi(z) \hat{u}_*^{p-1} hdz + \int_{\partial\Omega}\beta(z) \hat{u}_*^{p-1} hdz $$
	$$ = \int_{\Omega}[\hat{u}_*^{-\gamma} + \lambda^* f(z,\hat{u}_*)] hdz\ \mbox{for all } h \in W^{1,p}(\Omega)\ \mbox{(see \eqref{eq55},\ \eqref{eq60})}, $$
	$$ \Rightarrow \hat{u}_* \in S_{\lambda^*} \subseteq D_+\ \mbox{and so } \lambda^* \in {\mathcal L}. $$
The proof is now complete.
\end{proof}

	From this proposition it follows that
	$${\mathcal L} = (0,\lambda*]. $$

	The next bifurcation-type theorem summarizes our findings and provides a complete description of the dependence of the set of positive solutions of problem \eqref{eqp} on the parameter $\lambda >0$.

	\begin{theorem}\label{th16}
	If hypotheses $H(\xi),H(\beta),H_0,H(f)$ hold, then there exists $\lambda^* >0$ such that
	\begin{itemize}
	\item[(a)] for all $\lambda \in (0,\lambda^*)$ problem \eqref{eqp} has at least two positive solutions
	$$ u_0,\, \hat{u} \in D_+ ,\ u_0 \neq \hat{u};$$
	\item[(b)] for $\lambda = \lambda^*$ problem \eqref{eqp} has at least one positive solution $\hat{u}_*\in D_{+}$;
	\item[(c)] for all $\lambda > \lambda^*$ problem \eqref{eqp} does not have any positive solutions.
	\end{itemize}
	\end{theorem}

\section*{Acknowledgements.} This research was supported by the Slovenian Research Agency grants
P1-0292, J1-8131, J1-7025, N1-0064, and N1-0083.

\end{document}